\documentclass[11pt]{amsart}
\usepackage{hyperref}

\newcommand{\cX}{\mathcal{X}}

\renewcommand{\P}{\mathbb{P}}

\newcommand{\X}{\mathcal{X}}
\newcommand{\cI}{\mathcal{I}}

\newcommand{\cY}{\mathcal{Y}}

\newcommand{\T}{\mathbb{T}}

\newcommand{\R}{\mathbb{R}}
\newcommand{\E}{\mathbb{E}}
\newcommand{\N}{\mathbb{N}}

\newcommand{\Z}{\mathbb{Z}}

\newcommand{\norm}[1]{\left\Vert #1\right\Vert}

\theoremstyle{plain}
\newtheorem{theorem}{Theorem}[section]
\newtheorem{lemma}[theorem]{Lemma}
\newtheorem{proposition}[theorem]{Proposition}

\newtheorem*{theoremA'}{Theorem A'}
\newtheorem*{theoremB'}{Theorem B'}
\newtheorem*{theoremC'}{Theorem C'}

\newtheorem*{theorem*}{Theorem}

\newtheorem{corollary}[theorem]{Corollary}

\theoremstyle{definition}

\theoremstyle{remark}

\newtheorem*{remark}{Remark}

\begin{document}

\title[\tiny{Random sequences and pointwise convergence of multiple ergodic averages}]{Random sequences and pointwise convergence of multiple ergodic averages}

\author{N. Frantzikinakis}
\address[Nikos Frantzikinakis]{University of Crete, Department of mathematics, Knossos Avenue, Heraklion 71409, Greece} \email{frantzikinakis@gmail.com}

\author{E. Lesigne}
\address[Emmanuel Lesigne]{Laboratoire de Math\'ematiques et Physique Th\'eorique (UMR CNRS 6083)\\
Universit\'e Fran\c{c}ois Rabelais Tours \\
F\'ed\'eration de Recherche Denis Poisson\\
Parc de Grandmont \\
37200 Tours\\ France} \email{emmanuel.lesigne@lmpt.univ-tours.fr}

\author{M. Wierdl}
\address[M\'at\'e Wierdl]{Department of Mathematics\\
  University of Memphis\\
  Memphis, TN \\ 38152 \\ USA } \email{wierdlmate@gmail.com}

  \thanks{The  first author was partially supported by  Marie Curie IRG  248008 and the third author by NSF grant DMS-0801316.}



 \maketitle

\maketitle

\begin{abstract}
We prove  pointwise convergence, as $N\to \infty$, for the multiple
ergodic   averages $\frac{1}{N}\sum_{n=1}^N f(T^nx)\cdot
g(S^{a_n}x)$, where $T$ and  $S$ are commuting measure preserving
transformations, and $a_n$ is a random version of the sequence
$[n^c]$ for some appropriate $c>1$. We also prove similar mean
convergence results for  averages of the form
$\frac{1}{N}\sum_{n=1}^N f(T^{a_n}x)\cdot g(S^{a_n}x)$, as well as
pointwise results when $T$ and $S$ are powers of the same
transformation. The deterministic versions of these results, where
one replaces $a_n$ with $[n^c]$, remain open,
  and we hope that our method will indicate a fruitful way to approach these problems as well.
\end{abstract}

\setcounter{tocdepth}{1}

\section{Introduction}

\subsection{Background and new results} Recent advances in ergodic theory have sparked an outburst of activity in the study
of the limiting behavior of multiple ergodic averages.
 Despite the various successes in proving mean convergence results,
progress towards the corresponding pointwise convergence problems
has been very scarce.
 For instance, we still do not know   whether the averages
\begin{equation}\label{E:Av1}
\frac1N \sum_{n=1}^N f(T^nx)\cdot  g(S^nx)
\end{equation}
converge pointwise  when $T$ and $S$ are two commuting measure preserving transformations acting on the same probability space and $f$ and $g$ are bounded measurable functions. Mean convergence for such averages was shown in \cite{CL84} and  was recently generalized to an arbitrary number of   commuting transformations in \cite{Ta08}. On the other hand,  the situation with pointwise  convergence is much less satisfactory.
   Partial results that deal with special classes of transformations   can be found in
\cite{Bere85,Bere88,Les93a, LRR03, As05}. Without imposing any strictures
on the possible classes of the measure preserving transformations considered, pointwise
convergence is only known
  when $T$ and $S$ are powers of the same transformation \cite{Bou90} (see also \cite{D07} for an alternate proof), a result that has not been improved for
  twenty years.\\


More generally, for fixed $\alpha, \beta \in [1,+\infty)$, one would like to know whether the averages
 \begin{equation}\label{E:Av2}
\frac1N \sum_{n=1}^N f(T^{[n^\alpha]}x)\cdot  g(S^{[n^\beta]}x)
\end{equation}
converge pointwise. Mean convergence for
these and related averages has been
 extensively studied,
partly because of various links to questions in combinatorics.
In
particular, mean convergence is known when  $T=S$ and $\alpha$,
$\beta$ are positive integers \cite{HK05b,Lei05c}, or
positive  non-integers \cite{Fr09b}. Furthermore, for general
commuting transformations $T$ and $S$, mean convergence is known
when $\alpha$, $\beta$ are different positive integers \cite{CFH09}.
Regarding  pointwise convergence, again,
the situation is  much less satisfactory. When $\alpha,\beta$ are integers, some
 partial results for special classes of transformations
can be found in \cite{DL96}  and \cite{Lei05a}.
Furthermore, pointwise
convergence is known  for averages of the form
$\frac1N \sum_{n=1}^N f(T^{[n^\alpha]}x)$ with no restrictions on the transformation $T$
 (\cite{Bou88} for integers $\alpha$, and \cite{Wi89} or
\cite{BKQW05} for non-integers $\alpha$). But  for general commuting transformations $T$ and $S$, no  pointwise
convergence result is known, 
not even when $T=S$ and $\alpha\neq \beta$.

The main goal  of this article is to make some
progress related to the problem of pointwise convergence of the averages \eqref{E:Av2} 
by considering randomized versions of fractional powers of $n$, in
place of the deterministic ones,  for various suitably chosen exponents
$\alpha$ and $\beta$. In our first result, we study a variation of the
averages \eqref{E:Av2} where the iterates of $T$ are deterministic
and the iterates of $S$ are random.
 More precisely, we let  $a_n$ be a random version of the sequence $ [n^\beta]$ where $\beta\in (1, 14/13)$ is arbitrary.
We prove that almost surely  (the set of probability $1$ is universal)
the averages
\begin{equation}\label{E:Av3}
\frac1N \sum_{n=1}^N f(T^nx) \cdot g(S^{a_n}x)
\end{equation}
converge pointwise, and we determine the limit explicitly.
This is the first pointwise convergence result for
 multiple ergodic averages of the form $\frac1N \sum_{n=1}^N f(T^{a_n}x) \cdot g(S^{b_n}x)$,
 where $a_n,b_n$ are strictly increasing sequences and $T,S$ are general
  commuting measure preserving transformations. In fact,
even for mean convergence the result is new,
and  this is an instance where convergence of multiple ergodic averages
 involving  sparse iterates
 is obtained without the use of  rather deep
ergodic structure theorems and equidistribution results on nilmanifolds.

In our second result, we study a randomized version of the averages \eqref{E:Av2}
when $\alpha=\beta$.
In this case, we let   $ a_n $ be a random version of the sequence $[n^\alpha] $ where $\alpha\in (1, 2)$ is arbitrary,
 and prove that almost surely (the set of probability $1$ is universal) the averages
\begin{equation}\label{E:Av4}
\frac1N \sum_{n=1}^N f(T^{a_n}x) \cdot g(S^{a_n}x)
\end{equation}
 converge in the mean, and    conditionally to the pointwise convergence of the averages
\eqref{E:Av1}, they also converge pointwise.  Even for mean convergence, this gives  the first examples of sparse sequences of integers $a_n$ for which the averages
\eqref{E:Av4} converge for general commuting measure preserving transformations $T$ and $S$.

Because our convergence results come with explicit limit formulas, we can easily deduce  some related multiple recurrence
results. Using the correspondence principle of Furstenberg, these results  translate to statements in combinatorics about configurations that can be found in every  subset of  the integers, or  the integer lattice,
with positive upper density.

Let us also remark  that convergence of the averages \eqref{E:Av1} for not necessarily commuting transformations is known to fail in general\footnote{See example 7.1 in  \cite{Bere85},  or let
$T,S\colon \T\to\T$, given by $Tx=2x$, $Sx=2x+\alpha$, and $f(x)=e^{-2\pi i x}, g(x)=e^{2\pi i x}$,
 where $\alpha\in [0,1]$ is chosen so that the averages $\frac{1}{N}\sum_{n=1}^Ne^{2\pi i\cdot  2^n\alpha}$ diverge.}. We prove that this is also the  case for the averages \eqref{E:Av2}, \eqref{E:Av3}, and  \eqref{E:Av4}.

We state the exact results   in the next section,
where we also give precise definitions of the concepts used throughout
the paper.

\subsection{Precise statements of new results}
\subsubsection{Our setup} \label{SS:setup} We  work with  random sequences of integers that are
 constructed by selecting a positive  integer $n$ to be a member of our sequence
with probability $\sigma_n\in [0,1]$.
More precisely, let $(\Omega,\mathcal{F}, \mathbb{P})$ be a probability space, and  let $(X_n)_{n\in\N}$ be a sequence of independent random variables with
$$
\P(X_n=1)=\sigma_n  \ \text{ and } \ \P(X_n=0)=1-\sigma_n.
 $$ In the present article we always  assume that $\sigma_n= n^{-a}$ for some $a\in (0,1)$.
The random sequence $(a_n(\omega))_{n\in\N}$ is constructed by
taking the positive integers $n$ for which $X_n(\omega)=1$ in
increasing order. Equivalently, $a_n(\omega)$ is the smallest $k\in \N$ such that $X_1(\omega)+\cdots+X_k(\omega)=n$. We record the identity
 \begin{equation}\label{E:=n}
X_1(\omega)+\cdots+X_{a_n(\omega)}(\omega)=n
\end{equation}
for future use.

The  sequence $(a_n(\omega))_{n\in\N}$ is what we called random version of
the sequence $n^{1/(1-a)}$ in the previous subsection.
 Indeed, using a variation of the strong law of large numbers (see Lemma~\ref{L:estimate1} below), we have that almost surely
$\frac{1}{\sum_{k=1}^N\sigma_k}\sum_{k=1}^NX_k(\omega)$ converges to $1$.  Using the implied estimate   for $a_n(\omega)$ in place of $N$,  where  $n$ is suitably large, and \eqref{E:=n},  we deduce that almost surely
  $a_n(\omega)/n^{1/(1-a)}$  converges to a non-zero constant.


\subsubsection{Different iterates}
In our first result we study a randomized version of the averages \eqref{E:Av2} when
$\alpha=1$.

\begin{theorem}\label{T:main}
 With the notation of Section~\ref{SS:setup},  let   $\sigma_n= n^{-a}$ for some $a\in (0,1/14)$.
Then almost surely  the following holds: For every probability space $(X,\X,\mu)$,
commuting measure preserving transformations $T,S\colon X\to X$, and functions $f,g\in L^\infty(\mu)$,
for almost every $x\in X$ we have
  \begin{equation}\label{E:formula}
 \lim_{N\to\infty} \frac{1}{N}\sum_{n=1}^N f(T^nx)\cdot g(S^{a_n(\omega)}x)=\tilde{f}(x)\cdot \tilde{g}(x)
 \end{equation}
where
$\tilde{f}\mathrel{\mathop:}=\lim_{N\to\infty}\frac{1}{N}\sum_{n=1}^NT^nf= \E(f|\mathcal{I}(T))$,
$\tilde{g}\mathrel{\mathop:}= \lim_{N\to\infty}\frac{1}{N}\sum_{n=1}^NS^ng=\E(g|\mathcal{I}(S))$.\footnote{If $(X,\X,\mu)$ is a probability space, $f\in L^\infty(\mu)$,  and $\cY$ is a sub-$\sigma$-algebra of $\X$, we denote by $\E(f|\cY)$  the conditional expectation of $f$ given $\cY$. If $T\colon X\to X$ is a measure preserving transformation, by $\mathcal{I}(T)$ we denote the sub-$\sigma$-algebra of sets that are left invariant by $T$.}
\end{theorem}
We remark that the  conclusion of Theorem~\ref{T:main} can be easily
extended to all functions
$f\in L^p$, $g\in L^q$, where $p\in [1,+\infty]$ and  $q\in (1,+\infty]$ satisfy $1/p+1/q\leq 1$.\footnote{To see this, one uses a standard approximation argument and the fact that the averages $\frac{1}{N}\sum_{n=1}^NT^nf$ converge pointwise for $f\in L^p$ when $p\in [1,+\infty]$, and the same holds for  the averages $\frac{1}{N}\sum_{n=1}^NS^{a_n(\omega)}f$ for $f\in L^q$ when $q\in (1,+\infty]$ (see for example exercise 3 on page 78 of \cite{RW95}).}

Combining the limit formula  of Theorem~\ref{T:main}
with the  estimate (see Lemma 1.6 in \cite{Chu})
$$
\int f\cdot \E(f|\X_1)\cdot \E(f|\X_2)\ d\mu\geq \Big(\int f \ d\mu\Big)^3,
$$
that holds for every non-negative function $f\in L^\infty(\mu) $ and sub-$\sigma$-algebras $\X_1$ and $\X_2$ of $\X$, we  deduce  the following:
\begin{corollary}\label{C:234}
With the assumptions of Theorem~\ref{T:main}, we get almost surely, that for every $A\in \X$ we have
$$
 \lim_{N\to\infty} \frac{1}{N}\sum_{n=1}^N \mu(A\cap T^{-n}A\cap S^{-a_n(\omega)}A)\geq \mu(A)^3.
$$
\end{corollary}
The {\em
upper  density} $\bar{d}(E)$ of a set $E\subset\Z^2$ is defined
by $\bar{d}(E) = \limsup_{N \to\infty}\frac{|E\cap [-N,N]^2|}{|[-N,N]^2 |}$.
Combining the previous multiple recurrence result with
a multidimensional version of Furstenberg's
correspondence principle \cite{FuK79}, we deduce the following:
\begin{corollary}\label{C:Sz1}
With the notation of Section~\ref{SS:setup},  let   $\sigma_n= n^{-a}$ for some $a\in (0,1/14)$. Then almost surely,  for every
${\bf v}_1,{\bf v}_2\in \Z^2$ and
$E\subset\Z^2$ we have
$$
\liminf_{N\to\infty}\frac{1}{N}\sum_{n=1}^N \bar{d}\bigl(E\cap (E-n{\bf v}_1)\cap
 (E-a_n(\omega){\bf v}_2) \bigr)\geq (\bar{d}(E))^3.
$$
\end{corollary}
We remark that  in the previous statement we could have used the upper Banach density $d^*$ in place of the upper density $\bar{d}$. This is defined
by $d^*(E) = \limsup_{|I| \to\infty}\frac{|E\cap I|}{|I|}$, where $|I|$ denotes the area of a rectangle $I$ and   the $\limsup$ is taken over all rectangles of $\Z^2$ with side lengths that increase to infinity. The same holds for the statement of Corollary~\ref{C:Sz2} below.

\subsubsection{Same iterates}
In our next result we study a randomized version of the averages~\eqref{E:Av1}. By $Tf$ we denote the composition $f\circ T$.
\begin{theorem}\label{T:mainAP}
With the notation of Section~\ref{SS:setup}, let   $\sigma_n= n^{-a}$ for some $a\in (0,1/2)$.
Then almost surely  the following holds: For every probability space $(X,\X,\mu)$,
commuting measure preserving transformations $T,S\colon X\to X$, and functions $f,g\in L^\infty(\mu)$, the averages  \begin{equation}\label{E:formulaAP}
 \frac{1}{N}\sum_{n=1}^N T^{a_n(\omega)}f\cdot S^{a_n(\omega)}g
 \end{equation}
converge in $L^2(\mu)$
 and their limit  equals the $L^2$-limit of the averages  $\frac{1}{N}\sum_{n=1}^N T^{n}f\cdot S^{n}g$ (this exists by \cite{CL84}). Furthermore,   if $T$ and $S$ are powers of the same transformation, then  the averages \eqref{E:formulaAP} converge pointwise.
\end{theorem}
Our argument actually shows that   the averages \eqref{E:formulaAP} converge  pointwise if and only if the averages \eqref{E:Av1}   convergence pointwise. Furthermore, using our method, one can  get similar  convergence results
for  other  random multiple ergodic averages.
For instance, our method can be modified  and combined with the results from \cite{Ta08} and \cite{CFH09} to show that for every $\ell\in\N$,
if $\sigma_n=n^{-a}$ and $a$ is small enough (in fact any $a\in (0, 2^{-\ell})$ works),  then
 almost surely the averages
$\frac{1}{N}\sum_{n=1}^N T_1^{a_n(\omega)}f_1\cdots
T_\ell^{a_n(\omega)}f_\ell$ and  $\frac{1}{N}\sum_{n=1}^N
T_1^{a_n(\omega)}f_1\cdot T_2^{(a_n(\omega))^2}f_2\cdots
T_\ell^{(a_n(\omega))^\ell}f_\ell$ converge in the mean.

Combining Theorem~\ref{T:mainAP} with the  multiple recurrence result of Furstenberg and Katznelson \cite{FuK79}, we deduce the following:
\begin{corollary}\label{C:123}
With the assumptions of Theorem~\ref{T:mainAP}, we get almost surely,  that if $A\in \X$ has positive measure, then
$$
\lim_{N\to\infty} \frac{1}{N}\sum_{n=1}^N\mu(A\cap T^{-a_n(\omega)}A\cap S^{-a_n(\omega)}A)>0.
$$
\end{corollary}
Combining the previous multiple recurrence result with the correspondence
 principle of Furstenberg \cite{Fu81}, we deduce the following:
\begin{corollary}\label{C:Sz2}
With the notation of Section~\ref{SS:setup},  let   $\sigma_n= n^{-a}$ for some $a\in (0,1/2)$. Then almost surely, for every
${\bf v}_1,{\bf v}_2\in \Z^2$, and every
$E\subset\Z^2$ with $\bar{d}(E)>0$, we have
$$
\liminf_{N\to\infty} \frac{1}{N}\sum_{n=1}^N\bar{d}\big(E\cap (E+a_n(\omega){\bf v}_1)\cap (E+a_n(\omega){\bf v}_2)\big)>0.
$$

\end{corollary}

\subsubsection{Non-recurrence and non-convergence}

One may wonder whether the assumption that the transformations $T$ and $S$ commute
  can be removed from  the statements of  Theorems~\ref{T:main} and \ref{T:mainAP} and the related corollaries. It can definitely be weakened;  probably assuming that
  the group generated by $T$ and $S$ is nilpotent suffices, see for example \cite{BL02} where
mean convergence of the averages \eqref{E:Av1} is shown under such an  assumption.
On the other hand,
 constructions
of Berend (Ex 7.1 in \cite{Bere85}) and Furstenberg (page 40 in \cite{Fu81}) show that
Theorem~\ref{T:mainAP} and Corollary~\ref{C:123} are false if the assumption
that the transformations $T$ and $S$ commute is completely removed.
Next, we state a rather general result which implies that one has similar obstructions
when dealing with  Theorem~\ref{T:main} and Corollary~\ref{C:234}.

Given a probability space $(X,\X,\mu)$, we say that a measure preserving transformation $T\colon X\to X$
 is \emph{Bernoulli} if the measure preserving system $(X,\X,\mu, T)$ is isomorphic to a Bernoulli shift on finitely many symbols.
\begin{theorem}\label{T:NonRec-NonConv}
Let  $a,b\colon \N\to\Z\setminus\{0\}$ be  two injective sequences.
Then there exist   a probability space $(X,\X,\mu)$ and measure preserving transformations $T,S\colon X\to X$, both of them Bernoulli, such that
\begin{itemize}
  \item for some   $f,g\in L^\infty(\mu)$
the averages $\frac{1}{N}\sum_{n=1}^N\int T^{a(n)} f\cdot S^{b(n)}g \ d\mu$ diverge, and

  \item for some $A\in \mathcal{X}$ with $\mu(A)>0$ we have $\mu(T^{-a(n)}A\cap S^{-b(n)}A)=0$ for every $n\in \N$.
\end{itemize}
\end{theorem}
One can use a   variation of our argument  to extend Theorem~\ref{T:NonRec-NonConv} to sequences of bounded multiplicity, meaning sequences $(c(n))$ that satisfy
$\sup_{m\in\text{range}(c)}\#\{n\in\N\mid c(n)=m\}<+\infty$.
On the other hand, Theorem~\ref{T:NonRec-NonConv} cannot be extended to all sequences that take any given integer value a finite number of times.
For instance, it is not hard to show that the pair of sequences  $a(n)=[\sqrt{n}]$, $b(n)=n$, is good for (multiple) recurrence and mean convergence (see the proof of Theorem~2.7  in \cite{Fr09b}).
\subsubsection{Further directions}
The restrictions on the range of the eligible parameter $a$ in
Theorem~\ref{T:main},  Theorem~\ref{T:mainAP}, and the related
corollaries, appears to be  far from best possible.\footnote{Any improvement in
the range of the eligible parameter $a$ in the statement of
Proposition~\ref{P:main'} or Proposition~\ref{P:mainAP}, would give
corresponding improvements in the statement of  Theorem~\ref{T:main}
and  Theorem~\ref{T:mainAP} and the related corollaries.} In fact,
any $a<1$ is expected to work, but it seems that new techniques are
needed to prove this. This larger range of parameters is known to
work for pointwise convergence of the
 averages $\frac{1}{N}\sum_{n=1}^N f(T^{a_n(\omega)}x)$ (see \cite{Bos83} for mean convergence, \cite{Bou88} for pointwise, and \cite{RW95} for a survey of related results). Furthermore, when $\sigma_n=\sigma\in (0,1)$ for every $n\in \N$, it is not clear whether the conclusion of Theorem~\ref{T:main} related to pointwise convergence holds
   (see Theorem~4 in \cite{LPWR94} for a related negative pointwise convergence result).

 Regarding Theorem~\ref{T:main}, it seems very likely that similar
convergence results
hold when the iterates of the transformation $T$ are given by other ``good'' deterministic sequences, like polynomial sequences.
Our argument does not give such an extension because it relies crucially on the linearity of the
iterates of $T$. Furthermore, it seems likely that similar convergence results hold when the iterates
of $T$ and $S$ are both given by random versions of different fractional powers, chosen independently.
Again our present argument does not seem to apply to this case.


\subsection{General conventions and notation}
We use the symbol $\ll$  when
some expression is  majorized by a constant multiple of some  other expression. If this constant
depends on  the variables $k_1,\ldots, k_\ell$,
we write $\ll_{k_1,\ldots, k_\ell}$.  We say that $a_n\sim b_n$ if $a_n/b_n$ converges to a non-zero constant. We denote by $o_N(1)$
 a quantity that converges to zero when $N\to \infty$ and
all other parameters are fixed. We say that two sequences are
\emph{asymptotically equal} whenever convergence of one implies
convergence of the other and both limits coincide. If $(\Omega,
\mathcal{F}, \P)$ is a probability space, and $X$ is a random
variable, we set $\E_\omega(X)\mathrel{\mathop:}=\int X \ d\P$. We say that a property
holds almost surely if it holds outside of a set with probability
zero. We often suppress writing the variable $x$ when we refer to
functions and the variable $\omega$ when we refer to  random
variables or  random sequences. Lastly, the following notation will
be used throughout the article: $\N\mathrel{\mathop:}=\{1,2,\ldots\}$,
$Tf\mathrel{\mathop:}=f\circ T$, $e(t)\mathrel{\mathop:}=e^{2\pi i t}$.

\section{Convergence for independent random iterates}
In this section we prove   Theorem~\ref{T:main}.
Throughout, we use the notation introduced in Section~\ref{SS:setup} and we also
let
$$
Y_n\mathrel{\mathop:}=X_n-\sigma_n,\quad   W_N\mathrel{\mathop:}=\sum_{n=1}^N\sigma_n.
$$
We remark that if $\sigma_n=n^{-a}$ for some $a\in (0,1)$, then $W_N\sim N^{1-a}$.
\subsection{Strategy of the proof} \label{SS:strategy1} Roughly speaking, in order to prove  Theorem~\ref{T:main} we go through the following  successive  comparisons:
\begin{align*}
\frac{1}{N}\sum_{n=1}^N f(T^nx) \cdot g(S^{a_n(\omega)}x)
&\approx
\frac{1}{W_N} \sum_{n=1}^N X_n(\omega) \cdot  f(T^{X_1(\omega)+\cdots+X_n(\omega)}x)\cdot g(S^nx)\\
&\approx
\frac{1}{W_N} \sum_{n=1}^N \sigma_n \cdot  f(T^{X_1(\omega)+\cdots+X_n(\omega)}x)\cdot g(S^nx)
\\
&\approx \frac{1}{N} \sum_{n=1}^N    f(T^{X_1(\omega)+\cdots+X_n(\omega)}x)\cdot g(S^nx)\\
&\approx
 \tilde{g}(x) \cdot \frac{1}{N} \sum_{n=1}^N    f(T^{X_1(\omega)+\cdots+X_n(\omega)}x)\\
&\approx
\tilde{g}(x)\cdot \frac{1}{N} \sum_{n=1}^N    f(T^nx)\\
&\approx
\tilde{f}(x)\cdot \tilde{g}(x),
\end{align*}
where $A_N(\omega,x)\approx B_N(\omega,x)$ means that almost surely (the set of probability $1$ is universal), the expression $A_N(\omega,x)$ is asymptotically equal to $B_N(\omega,x)$
for almost every $x\in X$.
The second comparison is the most crucial  one; essentially one has to get good estimates for the $L^2$ norm of the averages
$\frac{1}{W_N} \sum_{n=1}^N (X_n(\omega)-\sigma_n) \cdot  T^{X_1(\omega)+\cdots+X_n(\omega)}f\cdot S^ng$.
We do this in two steps. First we use an elementary estimate of van der Corput twice  to get a bound that depends only on the
random variables $Y_n$, and then  estimate the resulting expressions using the independence of the variables $Y_n$.
Let us also mention that the fifth comparison follows immediately by applying the first three for $g=1$.
\subsection{A reduction}\label{SS:Reductions}
Our first goal is to reduce  Theorem~\ref{T:main} to proving the following result:
\begin{proposition}\label{P:main'}
Suppose that  $\sigma_n=n^{-a}$ for some $a\in (0,1/14)$ and let $\gamma>1$ be a real number.
Then  almost surely the  following holds: For every probability space $(X,\X,\mu)$,
commuting measure preserving transformations $T,S\colon X\to X$, and functions $f,g\in L^\infty(\mu)$,
 we have
  \begin{equation}\label{E:formula'}
\sum_{k=1}^{\infty}  \norm{\frac{1}{W_{[\gamma^k]}}
\sum_{n=1}^{[\gamma^k]} Y_n(\omega) \cdot
T^{X_1(\omega)+\cdots+X_n(\omega)}f\cdot S^ng}^2_{L^2(\mu)}<+\infty.
 \end{equation}
\end{proposition}

We are going to establish this reduction in the next subsections.


 \subsubsection{First step}\label{SS:Step1}
 We assume, as we may, that
both functions $|f|$ and $|g|$ are pointwise bounded by $1$ for all points in $X$.
 By \eqref{E:=n} for every $\omega\in \Omega$ and $x\in X$ we have
$$
 \frac{1}{N}\sum_{n=1}^N f(T^nx)\cdot g(S^{a_n(\omega)}x)=
 \frac{1}{N}\sum_{n=1}^Nf(T^{X_1(\omega)+\cdots+X_{a_n(\omega)}(\omega)}x)\cdot g(S^{a_n(\omega)}x).
 $$
A moment of reflection shows that for every bounded sequence $(b_n)_{n\in\N}$, for every $\omega \in \Omega$,  the averages
$$
\frac{1}{N}\sum_{n=1}^N b_{a_n(\omega)}
$$
and the averages
$$
\frac{1}{W_N(\omega)} \sum_{n=1}^N  X_n(\omega)\cdot b_n,
$$
where $W_N(\omega)\mathrel{\mathop:}= X_1(\omega)+\cdots+X_N(\omega)$,
 are asymptotically equal as $N\to\infty$.  Moreover, Lemma~\ref{L:estimate1} in the Appendix
gives  that almost surely
 $ \lim_{N\to\infty} W_N(\omega)/W_N= 1$. Therefore, the last averages are asymptotically equal to the averages
 $$
\frac{1}{W_N} \sum_{n=1}^N  X_n(\omega)\cdot b_n.
$$
Putting these observations together, we see that for almost every $\omega\in \Omega$ the averages in \eqref{E:formula}
and the averages
\begin{equation}\label{E:MainAverages}
\frac{1}{W_N} \sum_{n=1}^N X_n(\omega) \cdot  f(T^{X_1(\omega)+\cdots+X_n(\omega)}x)\cdot g(S^nx)
\end{equation}
are asymptotically equal  for  every $x\in X$.

\subsubsection{Second step}\label{SS:Step2}
Next, we  study the limiting behavior of the averages \eqref{E:MainAverages} when the random variables
$X_n$ are replaced by their  mean. Namely, we study the averages
\begin{equation}\label{E:reduction1}
\frac{1}{W_N} \sum_{n=1}^N \sigma_n \cdot  f(T^{X_1(\omega)+\cdots+X_n(\omega)}x)\cdot g(S^nx).
\end{equation}
By Lemma~\ref{L:ChangeVar} in the Appendix,
for  every $\omega\in \Omega$ and $x\in X$    they are asymptotically equal to the averages
\begin{equation}\label{E:reduction2}
\frac{1}{N} \sum_{n=1}^N f(T^{X_1(\omega)+\cdots+X_n(\omega)}x)\cdot g(S^nx).
\end{equation}

\begin{lemma}
Suppose that $\sigma_n=n^{-a}$ for some $a\in (0,1)$.
 Then almost surely the following holds:
 For every probability space $(X,\cX,\mu)$, measure preserving transformations  $T,S\colon X\to X$,
 and functions $f,g\in L^\infty(\mu)$, we have
$$
\lim_{N\to \infty}\Big(\frac{1}{N} \sum_{n=1}^N f(T^{X_1(\omega)+\cdots+X_n(\omega)}x)\cdot g(S^nx)
-\frac{1}{N} \sum_{n=1}^N f(T^{X_1(\omega)+\cdots+X_n(\omega)}x)\cdot \E(g|\cI(S))(x)\Big)=0
$$
for almost every $x\in X$.
\end{lemma}
\begin{proof}
It suffices to show that almost surely,
 if $\E(g|\cI(S))=0$, then   $\lim_{N\to\infty}A_N(f,g,\omega,x)= 0$ for almost every $x\in X$,
 where
$$
A_N(f,g,\omega,x)\mathrel{\mathop:}=
\frac{1}{N} \sum_{n=1}^N f(T^{X_1(\omega)+\cdots+X_n(\omega)}x)\cdot g(S^nx).
$$

First we consider functions $g$ of the form   $h-Sh$ where $h\in
L^\infty(\mu)$. Assuming, as we may,  that both $|f|$ and $|h|$ are
pointwise bounded by $1$ for all points in $X$, partial summation
gives that
$$
A_N(f,h-Sh,\omega,x)= \frac{1}{N}\sum_{n=1}^N\big(f(T^{X_1(\omega)+\cdots+X_n(\omega)}x)-
f(T^{X_1(\omega)+\cdots+X_{n-1}(\omega)}x)\big)\cdot
 h(S^nx) +o_{N}(1).
$$
The complex  norm of the last expression is bounded by a constant times the average
$$
\frac{1}{N}\sum_{n=1}^N{\bf 1}_{E_n}(\omega)
$$
 where  $E_n\mathrel{\mathop:}=\{\omega\colon X_n(\omega)=1\}$.
Since $\P(E_n)=n^{-a}$, combining our assumption with   Lemma~\ref{L:BorCan} in the Appendix, we get that the last average
 converges  almost surely  to $0$ as $N\to\infty$. Therefore,
on a set $\Omega_0$ of probability $1$, that depends only on the random variables $X_n$,   we have
\begin{equation}\label{E:sd1}
\lim_{N\to\infty} A_N(f,h-Sh,\omega,x)=0
\end{equation}
for almost every $x\in X$.

 Furthermore,
 using  the trivial estimate
$$
|A_N(f,g,\omega,x)|\leq \frac{1}{N}\sum_{n=1}^N|g|(S^nx),
$$
and then applying the pointwise ergodic theorem for the
transformation $S$, we get for
every $\omega\in \Omega$ that
\begin{equation}\label{E:sd2}
\int \limsup_{N\to\infty}|A_N(f,g,\omega,\cdot)| \ d\mu \leq
\norm{g}_{L^1(\mu)}.
\end{equation}

Since every function $g\in L^\infty(\mu)$ that satisfies $\E(g|\cI(S))=0$ can be approximated in $L^1(\mu)$
arbitrarily well by functions of the form $h-Sh$ with $h\in L^\infty(\mu)$,
combining \eqref{E:sd1} and \eqref{E:sd2}, we get for every $\omega \in \Omega_0$,
that if $\E(g|\cI(S))=0$, then   $\lim_{N\to\infty}A_N(f,g,\omega,x)=0$ for almost every $x\in X$.
This completes the proof.
\end{proof}

\subsubsection{Third step}\label{SS:Step3}
We next turn our attention to the study of the limiting behavior
  of the averages
\begin{equation}\label{E:3rdStep}
\frac{1}{N} \sum_{n=1}^N f(T^{X_1(\omega)+\cdots+X_n(\omega)}x).
\end{equation}

\begin{lemma}
Let   $\sigma_n=n^{-a}$ for some $a\in (0,1/14)$.
Then almost surely the following holds: For every probability space  $(X,\X,\mu)$, measure preserving transformation $T\colon X\to X$,  and function $f\in L^\infty(\mu)$,  the
averages in \eqref{E:3rdStep} converge to $\E(f|\mathcal{I}(T))(x)$ for almost every $x\in X$.
\end{lemma}
\begin{remark}
Improving the range of the parameter $a$  would not lead to
corresponding improvements in our main results. On the other hand,
the restricted range we used enables us to give a succinct proof
using Proposition~\ref{P:main'}.
\end{remark}
\begin{proof}
We assume, as we may, that the function $|f|$ is pointwise bounded
by $1$ for all points in $X$. First notice that by
Lemma~\ref{L:ChangeVar} in the Appendix, for every $\omega\in
\Omega$ and $x\in X$, the averages in \eqref{E:3rdStep} are
asymptotically equal to the averages
$$
 \frac{1}{W_N}\sum_{n=1}^{N}\sigma_n \cdot f(T^{X_1(\omega)+\cdots+X_n(\omega)}x)
$$
where $W_N\mathrel{\mathop:}=\sum_{n=1}^Nn^{-a}\sim N^{1-a}$.
Combining this observation with Corollary~\ref{C:lacunary} on the
Appendix, we deduce that it suffices to show that almost surely the
following holds:
  For every probability space  $(X,\X,\mu)$, measure preserving transformation $T\colon X\to X$,
  function $f\in L^\infty(\mu)$, and $\gamma\in \{1+1/k,k\in\N\}$, we have
\begin{equation}\label{E:3rdStep'}
\lim_{N\to\infty}\frac{1}{W_{[\gamma^N]}} \sum_{n=1}^{[\gamma^N]} \sigma_n \cdot f(T^{X_1(\omega)+\cdots+X_n(\omega)}x)= \E(f|\mathcal{I}(T))(x)
\end{equation}
for almost every $x\in X$.

Using Proposition~\ref{P:main'} for  $g=1$, we get that almost
surely (the set of probability $1$ depends only on the random
variables $X_n$), for every  $\gamma\in \{1+1/k,k\in\N\}$,   the
averages in \eqref{E:3rdStep'} are   asymptotically equal to the
averages
$$
\frac{1}{W_{[\gamma^N]}}\sum_{n=1}^{[\gamma^N]} X_n(\omega)\cdot  f(T^{X_1(\omega)+\cdots+X_n(\omega)}x)
$$
for  almost every $x\in X$. Hence, it suffices to study the limiting behavior of the averages
$$
\frac{1}{W_N}\sum_{n=1}^{N} X_n(\omega)\cdot  f(T^{X_1(\omega)+\cdots+X_n(\omega)}x).
$$
Repeating
the argument used in Section~\ref{SS:Step1} (with $g=1$),  we deduce that for every $\omega\in \Omega$ and $x\in X$,  they are asymptotically  equal to the averages
$$
\frac{1}{N}\sum_{n=1}^N  f(T^{X_1(\omega)+\cdots+X_{a_n(\omega)}(\omega)}x)=
\frac{1}{N} \sum_{n=1}^N f(T^nx)
$$
where the last equality follows from \eqref{E:=n}.
Finally, using the pointwise ergodic theorem we get  that the last averages converge to $\E(f|\mathcal{I}(T))(x)$ for almost every $x\in X$.
This completes the proof.
\end{proof}

\subsubsection{Last step}
We prove Theorem~\ref{T:main} by combining Proposition~\ref{P:main'}
with the arguments in the previous three steps.  We start with
Proposition~\ref{P:main'}. It gives  that there exists a set
$\Omega_0\in \mathcal{F}$ of probability $1$ such that  for every
$\omega\in \Omega_0$ the following holds: For every probability
space $(X,\X,\mu)$, commuting measure preserving transformations
$T,S\colon X\to X$, functions $f,g\in L^\infty(\mu)$, and $\gamma\in
\{1+1/k, k\in \N\}$, we have
 \begin{equation}\label{E:serconverges}
\sum_{N=1}^\infty
\norm{S_{[\gamma^N]}(\omega,\cdot)}^2_{L^2(\mu)}<+\infty
\end{equation}
where
$$
S_N(\omega,x)\mathrel{\mathop:}=\frac{1}{W_{N}} \sum_{n=1}^{N} Y_n(\omega) \cdot  f(T^{X_1(\omega)+\cdots+X_n(\omega)}x)\cdot g(S^nx).
 $$
 In the remaining argument $\omega$ is assumed to belong to the  aforementioned  set $\Omega_0$.
Notice that  \eqref{E:serconverges} implies that
 $$
\lim_{N\to\infty} S_{[\gamma^N]}(\omega,x)= 0 \qquad \text{for almost every } x\in X.
 $$

 We conclude that  for almost every $x\in X$, for  every $\gamma\in \{1+1/k, k\in \N\}$, the difference
$$
\frac{1}{W_{[\gamma^N]}} \sum_{n=1}^{[\gamma^N]} X_n(\omega) \cdot  f(T^{X_1(\omega)+\cdots+X_n(\omega)}x)\cdot g(S^nx)
- \frac{1}{W_{[\gamma^N]}} \sum_{n=1}^{[\gamma^N]} \sigma_n \cdot  f(T^{X_1(\omega)+\cdots+X_n(\omega)}x)\cdot g(S^nx)
 $$
 converges to $0$ as $N\to \infty$.
 In Sections~\ref{SS:Step2} and \ref{SS:Step3} we proved that   for almost every $x\in X$ we have
 $$
 \lim_{N\to\infty} \frac{1}{W_{N}} \sum_{n=1}^{N} \sigma_n \cdot  f(T^{X_1(\omega)+\cdots+X_n(\omega)}x)\cdot g(S^nx)=
 \tilde{f}(x)\cdot \tilde{g}(x),
 $$
where
$\tilde{f}\mathrel{\mathop:}= \E(f|\mathcal{I}(T))$, and
$\tilde{g}\mathrel{\mathop:}= \E(g|\mathcal{I}(S))$.
We deduce from the above that
 for almost every $x\in X$,  for every $\gamma\in \{1+1/k, k\in \N\}$, we have that
$$
\lim_{N\to\infty} \frac{1}{W_{[\gamma^N]}} \sum_{n=1}^{[\gamma^N]} X_n(\omega) \cdot  f(T^{X_1(\omega)+\cdots+X_n(\omega)}x)\cdot g(S^nx)
= \tilde{f}(x)\cdot \tilde{g}(x).
$$
Since the sequence $(W_n)$ satisfies the assumptions of Corollary~\ref{C:lacunary} in the Appendix,
we conclude that  for  non-negative functions $f,g\in L^\infty(\mu)$, for almost every $x\in X$, we have
\begin{equation}\label{E:sdf}
\lim_{N\to\infty}\frac{1}{W_{N}} \sum_{n=1}^{N} X_n(\omega) \cdot  f(T^{X_1(\omega)+\cdots+X_n(\omega)}x)\cdot g(S^nx)
= \tilde{f}(x)\cdot \tilde{g}(x).
\end{equation}
Splitting the real and imaginary part of the function $f$  as a difference of two non-negative functions,  doing the same for the function $g$,  and using the linearity
 of the operator $f\to\tilde{f}$, we  deduce that  \eqref{E:sdf} holds for arbitrary $f,g\in L^\infty(\mu)$.

Lastly, combining the previous identity and the argument used in Section~\ref{SS:Step1}, we deduce that for almost every $x\in X$ we have
$$
\lim_{N\to\infty}\frac{1}{N} \sum_{n=1}^{N}   f(T^nx)\cdot g(S^{a_n(\omega)}x)
= \tilde{f}(x)\cdot \tilde{g}(x).
$$
We have therefore established:
\begin{proposition}
If Proposition~\ref{P:main'} holds, then Theorem~\ref{T:main} holds.
\end{proposition}
In the next subsection we prove Proposition~\ref{P:main'}.


\subsection{Proof of Proposition~\ref{P:main'}} The proof of Proposition~\ref{P:main'} splits in two parts.
First we estimate  the $L^2$ norm of the averages
$\frac{1}{W_N}\sum_{n=1}^N Y_n  \cdot T^{X_1+\cdots+X_n}f\cdot S^ng$
 by an expression
that is independent of  the transformations $T,S$ and the  functions $f,g$.
 The main idea is to use van der Corput's Lemma (see Lemma~\ref{L:VDC2} in the Appendix) enough times to get the desired  cancelation, allowing enough flexibility on the parameters involved to ensure that certain terms become negligible. Subsequently, using moment estimates, we show that the resulting expression is almost surely summable along exponentially growing sequences of integers.

Before delving into the details we make some preparatory remarks that will help us ease our notation. We assume that both functions $f,g$ are bounded by $1$.
We remind the reader that
$$
\sigma_n=n^{-a}, \quad  W_N\sim N^{1-a}
$$
for some $a\in (0,1)$.
We are  going to use parameters $M$ and $R$ that satisfy
$$
 \quad M=[N^{b}], \quad R=[N^{c}]
$$
for some $b,c\in (0,1)$ at our disposal. We impose more restrictions on $a,b,c$ as we move on.

\subsubsection{Eliminating the dependence on the transformations and the functions}
To simplify our notation, in this subsection, when we write $\sum_{n=1}^{N^\alpha}$ we mean $\sum_{n=1}^{[N^\alpha]}$.

Using Lemma~\ref{L:VDC2} in the Appendix with $M=[N^b]$ and $v_n=Y_n\cdot  T^{X_1+\cdots+X_n}f\cdot S^ng$, we get that
\begin{equation}\label{E:A_N}
A_N\mathrel{\mathop:}=\norm{N^{-1+a}\sum_{n=1}^N Y_n\cdot  T^{X_1+\cdots+X_n}f\cdot S^ng}^2_{L^2(\mu)}\ll
A_{1,N}+A_{2,N},
\end{equation}
where
$$
A_{1,N}\mathrel{\mathop:}=N^{-1+2a-b}\cdot \sum_{n=1}^N\norm{ Y_n\cdot  T^{X_1+\cdots+X_n}f\cdot S^ng}^2_{L^2(\mu)}
$$
and
\begin{multline*}
A_{2,N}\mathrel{\mathop:}=N^{-1+2a-b} \cdot \sum_{m=1}^{N^b}  \Big| \sum_{n=1}^{N-m}\int
Y_{n+m}\cdot Y_n\cdot  T^{X_1+\cdots+X_{n+m}}f\cdot S^{n+m}g\cdot    T^{X_1+\cdots+X_n}\bar{f}
\cdot S^n\bar{g}
\ d\mu
\Big|.
\end{multline*}

 We estimate $A_{1,N}$. Since $\E(Y_n^2)=\sigma_n-\sigma_n^2\sim n^{-a}$, Lemma~\ref{L:estimate1} in the Appendix gives for every
 $a\in (0,1)$ that
 $\sum_{n=1}^N Y_n^2\sim \sum_{n=1}^N \E(Y_N^2)\sim N^{1-a}$. Therefore, almost surely we have
$$
A_{1,N}\ll N^{-1+2a-b} \sum_{n=1}^NY_n^2 \ll_\omega
 N^{-1+2a-b}\cdot N^{1-a}=N^{a-b}.
$$
It follows that $A_{1,N}$ is  bounded by a negative power of $N$ as long as
$$
b>a.
$$

Next, we estimate $A_{2,N}$. We compose with $S^{-n}$ and use the  Cauchy-Schwarz  inequality. We get
$$
A_{2,N}\ll N^{-1+2a-b} \cdot \sum_{m=1}^{N^b}  \norm{\sum_{n=1}^{N-m}
Y_{n+m}\cdot Y_n\cdot  S^{-n}T^{X_1+\cdots+X_{n+m}}f\cdot
S^{-n}T^{X_1+\cdots+X_n} \bar{f}}_{L^2(\mu)}.
$$
Furthermore, since
$$
N^{-1+2a-b} \sum_{m=1}^{N^b}m \ll N^{-1+2a+b},
$$
we get the estimate
$$
A_{2,N}\ll N^{-d_1}+ N^{-1+2a-b} \cdot \sum_{m=1}^{N^b}  \norm{\sum_{n=1}^{N}
Y_{n+m}\cdot Y_n\cdot  S^{-n}T^{X_1+\cdots+X_{n+m}}f\cdot
S^{-n}T^{X_1+\cdots+X_n}\bar{f}}_{L^2(\mu)}
$$
where $d_1\mathrel{\mathop:}=1-2a-b$ is positive
 as long as
$$
2a+b<1.
$$ Using the  Cauchy-Schwarz  inequality we get
$$
A_{2,N}^2\ll N^{-2d_1}+N^{-2+4a-b} \cdot \sum_{m=1}^{N^b}  \norm{\sum_{n=1}^{N}
Y_{n+m}\cdot Y_n\cdot  S^{-n}T^{X_1+\cdots+X_{n+m}}f\cdot   S^{-n}T^{X_1+\cdots+X_n}\bar{f}}^2_{L^2(\mu)}.
$$
Next  we use Lemma~\ref{L:VDC2} in the Appendix with $R=[N^c]$ and the obvious choice of functions $v_n$,  in order to estimate the square of the $L^2$ norm above.  We
get the estimate
$$
A_{2,N}^2\ll N^{-2d_1}+ A_{3,N}+A_{4,N},
$$
where $A_{3,N}$, $A_{4,N}$, can be computed as before.
  Using Lemma~\ref{L:estimate2} in the Appendix, and the estimate $\E(Y_n^2)\sim n^{-a}$,  we deduce that  almost surely,
  for every $a\in (0,1/6)$ we have
$$
A_{3,N}\ll N^{-1+4a-b-c}\sum_{m=1}^{N^b}  \sum_{n=1}^{N}Y^2_{n+m}Y^2_n\ll_\omega
N^{2a-c}=N^{-d_2}
$$
where  $d_2>0$ as long as
$$
2a<c.
$$
Composing with $T^{-(X_1+\cdots+X_n)}S^{n}$, using that $T$ and $S$ commute,  and  the Cauchy-Schwarz  inequality, we see that
\begin{multline*}
A_{4,N}\ll N^{-1+4a-b-c} \cdot \sum_{m=1}^{N^b} \sum_{r=1}^{N^c}  \Big\|\sum_{n=1}^{N-r}
Y_{n+m+r}\cdot Y_{n+r} \cdot  Y_{n+m}\cdot Y_n\cdot \\
T^{X_{n+1}+\cdots+X_{n+m+r}}S^{-r}f\cdot   T^{X_{n+1}+\cdots+X_{n+r}}S^{-r}\bar{f}
\cdot  T^{X_{n+1}+\cdots+X_{n+m}}\bar{f}\Big\|_{L^2(\mu)}.
\end{multline*}
Since for every  $k\in \N$ we have $X_{n+1}+\cdots+X_{n+k}\in \{0,\ldots, k\}$, it follows that
\begin{multline}\label{E:A4}
A_{4,N}\ll A_{5,N}\mathrel{\mathop:}= N^{-1+4a-b-c} \cdot \sum_{m=1}^{N^b} \sum_{r=1}^{N^c}
\sum_{k_1=0}^{m+r}\sum_{k_2=0}^r \sum_{k_3=0}^m \Big|\sum_{n=1}^{N-r}
Y_{n+m+r}\cdot Y_{n+r} \cdot  Y_{n+m}\cdot Y_n\cdot \\
{\bf 1}_{\sum_{i=1}^{m+r}X_{n+i}=k_1}(n)\cdot {\bf
1}_{\sum_{i=1}^rX_{n+i}=k_2}(n)\cdot {\bf 1}_{\sum_{i=1}^m
X_{n+i}=k_3}(n) \Big|.
\end{multline}
Summarizing,  we have just shown that as long as
\begin{equation}\label{E:Conditions1-3}
a<b, \qquad 2a+b<1, \qquad 2a< c, \qquad a\in (0,1/6), \qquad b,c\in (0,1),
\end{equation}
 almost surely the following holds: For every probability space $(X,\X,\mu)$,
commuting measure preserving transformations $T,S\colon X\to X$, and functions $f,g\in L^\infty(\mu)$
with $\norm{f}_{L^\infty(\mu)}\leq 1$ and $\norm{g}_{L^\infty(\mu)}\leq 1$,
we have
\begin{equation}\label{E:Est1-4}
A_N\ll_\omega N^{-d_3} +A_{5,N}
\end{equation}
for some $d_3>0$, where $A_{5,N}$ is defined in \eqref{E:A4}. Notice that  the expression $A_{5,N}$ depends only on the random variables $X_n$. Therefore,    in order to complete the proof of Proposition~\ref{P:main'}, it suffices to show that almost surely $A_{5,N}$ is  summable along exponentially growing sequences of integers.

\subsubsection{Estimating $A_{5,N}$ (End of proof of Proposition~\ref{P:main'}).}
Assuming that
\begin{equation}\label{E:Condition3'}
b<c,
\end{equation}
we get that
\begin{equation}\label{E:A4'}
\E_\omega(A_{5,N})\leq  N^{-1+4a-b-c}
 \sum_{m=1}^{N^b} \sum_{r=1}^{N^c}
\sum_{k_1=0}^{2N^c}\sum_{k_2=0}^{N^c} \sum_{k_3=0}^{N^b}
\E_\omega\left|\sum_{n=1}^{N-r}
Y_n\cdot Z_{n,m,r,k_1,k_2,k_3}
\right|
\end{equation}
where
$$
 Z_{n,m,r,k_1,k_2,k_3}\mathrel{\mathop:}=Y_{n+m+r}\cdot Y_{n+r} \cdot  Y_{n+m}\cdot
{\bf 1}_{\sum_{k=1}^{m+r}X_{n+k}=k_1}(n)\cdot
{\bf 1}_{\sum_{k=1}^rX_{n+k}=k_2}(n)\cdot {\bf 1}_{\sum_{k=1}^m X_{n+k}=k_3}(n).
$$
Using the Cauchy-Schwarz inequality we get
\begin{equation}\label{E:expectation}
\E_\omega\left|\sum_{n=1}^{N-r}
Y_n\cdot Z_{n,m,r,k_1,k_2,k_3}
\right|\leq \left(\E_\omega\left|\sum_{n=1}^{N-r}
Y_n\cdot Z_{n,m,r,k_1,k_2,k_3}
\right|^2\right)^{1/2}.
\end{equation}
We expand the square in order to compute its expectation. It is equal to
$$
\sum_{1\leq n_1,n_2\leq N-r}\E_\omega
(Y_{n_1}\cdot Z_{n_1,m,r,k_1,k_2,k_3}\cdot
Y_{n_2}\cdot Z_{n_2,m,r,k_1,k_2,k_3}).
$$
Notice that
   if $n_1<n_2$, then for every $m,r\in \N$, and non-negative integers $k_1,k_2,k_3$, the random variable $Y_{n_1}$ is independent of the variables  $Y_{n_2}$,  $Z_{n_1,m,r,k_1,k_2,k_3}$, and  $Z_{n_2,m,r,k_1,k_2,k_3}$.
   Since $Y_n$ has zero mean, it follows that if $n_1\neq n_2$, then
$$
\E_\omega
(Y_{n_1}\cdot Z_{n_1,m,r,k_1,k_2,k_3}\cdot
Y_{n_2}\cdot Z_{n_2,m,r,k_1,k_2,k_3})=0.
$$
Therefore, the right hand side of equation \eqref{E:expectation} is
equal to
$$
\left(\sum_{n=1}^{N-r}
\E_\omega(Y_n^2) \cdot  \E_\omega( Z_{n,m,r,k_1,k_2,k_3}^2)
\right)^{1/2}\leq
\left(\sum_{n=1}^{N-r}
\E_\omega(Y_n^2) \cdot  \E_\omega( Y_{n+m}^2\cdot Y_{n+r}^2\cdot Y_{n+m+r}^2)
\right)^{1/2}.
$$
If $r,m, n$ are fixed and $r\neq m$, then the variables $Y_{n+m}^2, Y_{n+r}^2,Y_{n+m+r}^2$ are independent, and as a consequence  the right hand side  is almost surely bounded  by
$$
 \left(\sum_{n=1}^{N}
\sigma_n^4
\right)^{1/2}\ll N^{1/2-2a}.
$$
On the other hand, if $r,m, n$ are fixed and $r=m$, then the random variables $Y_{n+r}^4, Y_{n+2r}^2$ are independent, and as a consequence
the right hand side is almost surely bounded  by
$$
 \left(\sum_{n=1}^{N}
\sigma_n^3
\right)^{1/2}\ll N^{1/2-3a/2}.
$$
Combining these two estimates with \eqref{E:A4'}, we deduce that
$$
\E_\omega(A_{5,N})\ll N^{-1+4a-b-c}(N^{1/2-2a+ 2b+3c}+N^{1/2-3a/2+2b+2c})
  =N^{-1/2+2a+b+2c}+N^{-1/2+5a/2+b+c}.
$$
For fixed $\varepsilon>0$,
letting   $a\in (0,1/6)$,  $b$ be greater and very close to $a$, and $c$ be greater  and very close to
$2a$, we get that the conditions \eqref{E:Conditions1-3} and   \eqref{E:Condition3'}
are satisfied, and
\begin{equation}\label{E:A4''}
\E_\omega(A_{5,N})\ll  N^{(-1+14a)/2+\varepsilon}+N^{(-1+11a)/2+\varepsilon}=N^{-d_4}
\end{equation}
for some $d_4$ that satisfies
\begin{equation}\label{E:d_4}
 d_4> (1-14a)/2-\varepsilon.
 \end{equation}
Therefore, for every $a\in (0,1/14)$, if $\varepsilon$ is small enough,  then the estimates \eqref{E:Est1-4} and \eqref{E:A4''} hold for some  $d_3, d_4>0$.

 Equation \eqref{E:A4''}  gives that  for every $\gamma>1$ we have
$$
 \sum_{N=1}^\infty \E_\omega(A_{5,{[\gamma^N]}}) <+\infty.
$$
As a consequence, for every $\gamma>1$ we have almost surely that
\begin{equation}\label{E:A5}
 \sum_{N=1}^\infty A_{5,{[\gamma^N]}}(\omega)  <+\infty.
\end{equation}
Recalling the definition of $A_N$ in \eqref{E:A_N}, and combining \eqref{E:Est1-4} and \eqref{E:A5}, we get that for every $a\in (0,1/14)$ and $\gamma>1$,
 almost surely the following holds: For every probability space $(X,\X,\mu)$,
commuting measure preserving transformations $T,S\colon X\to X$, and $f,g\in L^\infty(\mu)$, we have
 $$
\sum_{N=1}^\infty \norm{ S_{[\gamma^N]}(\omega,\cdot)
}^2_{L^2(\mu)}<+\infty
 $$
 where
$$
S_N(\omega,\cdot)\mathrel{\mathop:}=N^{-1+a}\sum_{n=1}^N Y_n(\omega)\cdot  T^{X_1(\omega)+\cdots+X_n(\omega)}f\cdot S^ng.
$$
This finishes the proof of Proposition~\ref{P:main'}.

\section{Convergence for the  same random iterates}
In this section we prove   Theorem~\ref{T:mainAP}. Throughout, we use the notation introduced in Section~\ref{SS:setup} and the beginning of Section~\ref{SS:Reductions}.
\subsection{Strategy of the proof}
 In order to prove  Theorem~\ref{T:mainAP} we go through the following  successive
   comparisons:
\begin{align*}
\frac{1}{N}\sum_{n=1}^N f(T^{a_n(\omega)}x) \cdot g(S^{a_n(\omega)}x)
&\approx
\frac{1}{W_N} \sum_{n=1}^N X_n(\omega) \cdot  f(T^nx)\cdot g(S^nx)\\
&\approx
\frac{1}{W_N} \sum_{n=1}^N \sigma_n \cdot  f(T^nx)\cdot g(S^nx)
\\
&\approx
\frac{1}{N} \sum_{n=1}^N  f(T^nx)\cdot g(S^nx),
\end{align*}
where our notation was explained in Section~\ref{SS:strategy1}.
The key comparison is the second. One needs to get good estimates for the $L^2$ norm of the averages
$\frac{1}{W_N} \sum_{n=1}^N Y_n(\omega) \cdot  T^{n}f\cdot S^ng$, where $Y_n\mathrel{\mathop:}=X_n-\sigma_n$.
We do this in two steps. First we use van der Corput's estimate  and Herglotz's theorem   to get a bound that depends only on the
random variables $Y_n$.
The resulting expressions turn out to be random trigonometric polynomials that can be  estimated  using classical techniques.\footnote{A faster way to get such an estimate is to apply van der Corput's Lemma twice.
 The drawback of this method is that
 the resulting expression converges to zero only when $\sigma_n=n^{-a}$ for some $a\in (0,1/4)$.}

\subsection{A reduction}
Arguing as in Section~\ref{SS:Reductions} (in fact the argument is much simpler in the current  case) we reduce  Theorem~\ref{T:mainAP} to proving the following result:
\begin{proposition}\label{P:mainAP}
Suppose that  $\sigma_n=n^{-a}$ for some $a\in (0,1/2)$ and let $\gamma>1$ be a real number.
Then  almost surely the  following holds: For every probability space $(X,\X,\mu)$,
commuting measure preserving transformations $T,S\colon X\to X$, and functions $f,g\in L^\infty(\mu)$, we have
  \begin{equation}\label{E:formulaAP'}
\sum_{k=1}^\infty  \norm{\frac{1}{W_{[\gamma^k]}}
\sum_{n=1}^{[\gamma^k]}  Y_n(\omega) \cdot  T^nf\cdot
S^ng}^2_{L^2(\mu)}<+\infty
 \end{equation}
 where $W_N\mathrel{\mathop:}=\sum_{n=1}^N\sigma_n$.
\end{proposition}
We prove this result in the next subsection.

\subsection{Proof of Proposition~\ref{P:mainAP}.} As was the case with the proof of
Proposition~\ref{P:main'} the proof of Proposition~\ref{P:mainAP} splits in two parts.

\subsubsection{Eliminating the dependence on the transformations and the functions}\label{SS:VdCAP}
We assume that both functions $f,g$ are bounded by $1$.
We start by using Lemma~\ref{L:VDC2} for $M=N$  and $v_n\mathrel{\mathop:}=Y_n\cdot  T^nf\cdot S^ng$
(this is essentially the ordinary expansion of the square of the sum). We get that
\begin{equation}\label{E:A_N-AP}
A_N\mathrel{\mathop:}=\norm{N^{-1+a}\sum_{n=1}^N Y_n\cdot  T^nf\cdot S^ng}^2_{L^2(\mu)}\ll
A_{1,N}+A_{2,N}
\end{equation}
where
$$
A_{1,N}\mathrel{\mathop:}=N^{-2+2a}\cdot \sum_{n=1}^N\norm{ Y_n\cdot  T^nf\cdot S^ng}^2_{L^2(\mu)}
$$
and
$$
A_{2,N}\mathrel{\mathop:}=N^{-2+2a} \cdot \sum_{m=1}^{N}  \left| \sum_{n=1}^{N-m}\int
Y_{n+m}\cdot Y_n\cdot  T^{n+m}f\cdot S^{n+m}g\cdot   T^n\bar{f} \cdot S^n\bar{g}
\ d\mu \right|.
$$

 We estimate $A_{1,N}$. Since $\E_\omega(Y_n^2)\sim n^{-a}$, Lemma~\ref{L:estimate1} gives
 $\sum_{n=1}^N Y_n^2\ll_\omega N^{1-a}$. It follows that almost surely we have
\begin{equation}\label{E:Est1-AP}
A_{1,N}\ll N^{-2+2a} \sum_{n=1}^NY_n^2 \ll_\omega
 N^{-2+2a}\cdot N^{1-a}=N^{a-1}.
\end{equation}
Therefore, $A_{1,N}$ is bounded by a negative power of $N$  for every $a\in (0,1)$.

We estimate $A_{2,N}$. Composing with $S^{-n}$ and using the Cauchy-Schwarz inequality we get
$$
A_{2,N}\ll N^{-2+2a} \cdot \sum_{m=1}^{N}  \norm{\sum_{n=1}^{N-m}
Y_{n+m}\cdot Y_n\cdot  S^{-n}T^{n+m}f\cdot   S^{-n}T^{n}\bar{f}}_{L^2(\mu)}.
$$
Using that $T$ and $S$ commute and  letting  $R=TS^{-1}$ and
$f_m=T^mf\cdot \bar{f}$, we rewrite the previous estimate as
$$
A_{2,N}\ll N^{-2+2a} \cdot \sum_{m=1}^{N}  \norm{\sum_{n=1}^{N-m}
Y_{n+m}\cdot Y_n\cdot  R^nf_m}_{L^2(\mu)}.
$$
Using  Herglotz theorem on positive definite sequences, and the fact
that the functions $f_m$ are bounded by $1$, we get that the
right hand side is bounded by a constant multiple of
$$
A_{3,N}\mathrel{\mathop:}=N^{-1+2a} \cdot \max_{1\leq m\leq N} \max_{t\in [0,1]}\Big|\sum_{n=1}^{N-m}
Y_{n+m}\cdot Y_n\cdot  e(nt)\Big|.
$$

Summarizing, we have shown that
\begin{equation}\label{E:A_{3,N}}
A_N\ll N^{a-1}+A_{3,N}.
\end{equation}
Therefore, in order to prove Proposition~\ref{P:mainAP} it remains to show that almost surely
 $A_{3,N}\ll_\omega N^{-d}$ for some $d>0$.
We do this in the next subsection.

\subsubsection{Estimating $A_{3,N}$ (End of proof of Proposition~\ref{P:mainAP}).}
The goal of this section is to prove the following result:
\begin{proposition}\label{P:UnifEst}
Suppose that  $\sigma_n\sim  n^{-a}$ for some $a\in (0,1/2)$.
 Then almost surely we have
$$
\max_{1
\leq m
\leq N}\max_{t\in [0,1]}\Big|\sum_{n=1}^{N-m}
Y_{n+m}\cdot Y_n\cdot  e(nt)\Big|\ll_\omega \ N^{1/2-a} \sqrt{\log N},
$$
\end{proposition}
Notice that by combining this estimate with \eqref{E:A_{3,N}}
we get a proof of Proposition~\ref{P:mainAP}, and as a consequence a proof of Theorem~\ref{T:mainAP}.

The key ingredient in the proof of Proposition~\ref{P:UnifEst} is the following lemma. It is a
strengthening  of an estimate of Bourgain \cite{Bou88} regarding random trigonometric polynomials. The proof of the lemma is a variation on the classical Chernoff's inequality (see e.g. Theorem 1.8 in \cite{Tao-Vu}), combined with  an elementary estimate on the uniform norm of a trigonometric polynomial.
We were motivated to use this argument, over the  one given  in
\cite{Bou88}, after reading a paper of Fan and Schneider (in particular, the proof of Theorem~6.4 in \cite{FS10}).

\begin{lemma}\label{L:BasicAP}
Let $(Z_{m,n})_{m,n\in\N}$ be a family of random variables, uniformly bounded by $1$, and with mean zero. Suppose that, for each fixed $m$, the random variables $Z_{m,n}, n\geq1,$ are independent. Let $(\rho_n)$ be a sequence of positive numbers such that $$\sup_{m\in \N}\big({\mathrm {Var}}(Z_{m,n})\big)\leq\rho_n\quad\text{and}\quad\lim_{N\to\infty}\frac1{\log N}\sum_{n=1}^N\rho_n=+\infty.
$$
Then, almost surely, we have
$$
\max_{1\leq m\leq N}\max_{t\in[0,1]}\left|\sum_{n=1}^N Z_{m,n} \cdot e(nt)\right|\ll_\omega \left(\sqrt{\log N\cdot\sum_{n=1}^N\rho_n}\right).
$$

\end{lemma}
\begin{proof}
It suffices to get the announced estimate for
$$
M_N\mathrel{\mathop:}=\max_{1\leq m\leq N}\max_{t\in [0,1]} |P_{m,N}(t)|
$$
where
$$
P_{m,N}(t)\mathrel{\mathop:}=\sum_{n=1}^N
Z_{m,n}\cdot \cos(2\pi nt).
$$
In a similar way we get an estimate with $\sin(2\pi nt)$ in place of $\cos(2\pi nt)$.

Since $|Z_{m,n}|\leq1$ and $\E_\omega(Z_{m,n})=0$, we have $\E_\omega\left(e^{\lambda Z_{m,n}}\right)\leq e^{\lambda^2\text{Var}(Z_{m,n})}$ for all $\lambda\in[-1,1]$.
(See Lemma 1.7 in \cite{Tao-Vu}.)

For every $m\in \N$,  $\lambda \in [-1,1]$,  and $t\in[0,1]$, we get  that
\begin{equation}\label{E:11}
\E_\omega\left(e^{\lambda P_{m,N}(t)}\right)=
\prod_{n=1}^N\E_\omega\left(e^{\lambda Z_{m,n} \cos(2\pi nt)}\right)\leq
\prod_{n=1}^N e^{(\lambda\cos(2\pi nt))^2\text{Var}(Z_{m,n})}
\leq  e^{ \lambda^2 R_N}
\end{equation}
where $$
R_N\mathrel{\mathop:}=\sum_{n=1}^N\sigma_n.
$$
Next
notice that for $\lambda\in [0,1]$ we have
\begin{equation}\label{E:22}
\E_\omega(e^{\lambda M_N})=\E_\omega\big( \max_{1\leq m\leq N} e^{\lambda \max_t |P_{m,N}(t)|} \big)\leq
\E_\omega\Big( \sum_{m=1}^{N} e^{\lambda \max_t |P_{m,N}(t)|} \Big)\leq
 N \max_{1\leq m\leq N}
\E_\omega(e^{\lambda M_{m,N}})
\end{equation}
where
$$
M_{m,N}\mathrel{\mathop:}=\max_{t\in [0,1]}{|P_{m,N}(t)|}.
$$

It is easy to see (e.g. Proposition 5 in Chapter 5, Section 2 of
\cite{Kah85}) that   there exist  random intervals $I_{m,N}$ of
length $|I_{m,N}|\geq N^{-2}$  such  that $|P_{m,N}(t)|\geq
M_{m,N}/2$ for every $t\in I_{m,N}$. Using this, we get
 that
\begin{multline*}
 \E_\omega(e^{\lambda_N M_{m,N}/2})\ll  N^2 \cdot \E_\omega\Big(\int_{I_{m,N}} (e^{\lambda_N P_{m,N}(t)}+e^{-\lambda_N P_{m,N}(t)}) \ dt\Big)\leq \\
 N^2 \cdot \E_\omega\Big(\int_{[0,1]} (e^{\lambda_N P_{m,N}(t)}+e^{-\lambda_N P_{m,N}(t)}) \ dt\Big)
\end{multline*}
where $\lambda_N \in [0,1]$ are numbers at our disposal. Using
\eqref{E:11} we get that
$$
\E_\omega\Big(\int_{[0,1]} (e^{\lambda_N P_{m,N}(t)}+e^{-\lambda_N
P_{m,N}(t)}) \ dt\Big)= \int_{[0,1]} \E_\omega\big( e^{\lambda_N
P_{m,N}(t)}+e^{-\lambda_N P_{m,N}(t)}\big) \ dt\leq 2 e^{R_N
\lambda_N^2}.
$$
Therefore,
$$
\E_\omega(e^{\lambda_N M_{m,N}/2})\ll N^2 \cdot  e^{R_N \lambda_N^2}.
$$
Combining this estimate  with \eqref{E:22}, we get
$$
\E_\omega(e^{\lambda_N M_{N}/2})\ll N^3 \cdot  e^{R_N \lambda_N^2}.
$$
Therefore, there exists a universal constant $C$ such that
$$
\E_\omega\Big(e^{\lambda_N/2 (M_{N}-2R_N \lambda_N -2\log(C N^5)\lambda_N^{-1})}\Big)\leq \frac{1}{N^2}.
$$
As a consequence,
\begin{equation}\label{E:2}
\P \big( M_N\geq 2R_N \lambda_N+2\log(CN^5)\lambda_N^{-1}\big) \leq \frac{1}{N^2}.
\end{equation}
For $\alpha,\beta$ positive, the function
$f(\lambda)=\alpha\lambda+\beta\lambda^{-1}$ achieves a minimum
$\sqrt{\alpha \beta}$ for $\lambda=\sqrt{\beta/\alpha}$. So letting
$\lambda_N=\sqrt{\log(CN^5)/(AR_N)}$ (by assumption $\lambda_N$  converges
to $0$, so $\lambda_N<1$  for large $N$) in \eqref{E:2} gives
$$
\P \Big( M_N\geq \sqrt{4 R_N\log(CN^5)}\Big) \leq \frac{1}{N^2}.
$$
By the Borel-Cantelli Lemma, we get almost surely that
$$
M_N\ll_\omega \sqrt{R_N\log N}.
$$
 This completes the proof.
\end{proof}
Finally we use Lemma~\ref{L:BasicAP} to prove Proposition~\ref{P:UnifEst}.
\begin{proof}[Proof of Proposition~\ref{P:UnifEst}]
Our goal is to apply Lemma~\ref{L:BasicAP} for the random variables
$Y_{n+m}\cdot Y_n$ where $Y_n=X_n-\sigma_n$. These random variables are bounded by $1$ and   have zero mean.
 We just have  to take some care because they are not independent.
We divide the positive integers  into two classes:
$$
\Lambda_{1,m}\mathrel{\mathop:}=\{n\colon 2km< n\leq (2k+1)m \ \text{ for some non-negative integer } k\}
$$
and
$$
\Lambda_{2,m}\mathrel{\mathop:}=\{n\colon (2k+1)m< n\leq (2k+2)m \ \text{ for    some non-negative integer } k\}.
$$
Then for fixed $m\in \N$,  the random variables $Y_{n+m}\cdot Y_n$, $n\in
\Lambda_{1,m}$, are independent, and the same holds for the random
variables $Y_{n+m}\cdot Y_n$, $n\in \Lambda_{2,m}$.
For $i=1,2$, we  apply Lemma~\ref{L:BasicAP} to the random variables $$Z_{m,n}\mathrel{\mathop:}=Y_{n+m}\cdot Y_n\cdot {\bf1}_{\Lambda_{i,m}\cap[1,N-m]}(n).$$
Notice that  either $\text{Var}(Z_{m,n})=0$, or $$\text{Var}(Z_{m,n})=\sigma_{n+m}\sigma_n-\sigma_{n+m}^2\sigma_n-\sigma_{n+m}\sigma_n^2+\sigma_{n+m}^2\sigma_n^2\leq \sigma_{n+m}\sigma_n\leq\sigma_n^2 \sim n^{-2a}.$$
Since  $\sum_{n=1}^Nn^{-2a} \sim N^{1-2a}$ and $a<1/2$, the assumptions of  Lemma~\ref{L:BasicAP} are satisfied for $\rho_n=n^{-2a}$.
   We deduce that  almost surely we have
 $$
\max_{1
\leq m
\leq N}\max_{t\in [0,1]}\Big|\sum_{n=1}^{N-m}
Y_{n+m}\cdot Y_n\cdot  e(nt)\Big|\ll_\omega  N^{1/2-a}\sqrt{\log N}.
$$
This completes the proof.
\end{proof}

\section{Non-recurrence and non-convergence.}
In this section we prove Theorem~\ref{T:NonRec-NonConv}. The proof is based on the following lemma:
\begin{lemma}\label{L:NonRecConv}
Let $a,b\colon \N\to\Z\setminus\{0\}$ be  injective sequences and $F$ be any subset of $\N$.
Then there exist   a probability space $(X,\X,\mu)$,  measure preserving transformations $T,S\colon X\to X$, both of them Bernoulli, and $A\in \X$, such that
$$
\mu\big(T^{-a(n)}A\cap S^{-b(n)}A\big)=\begin{cases} 0&\ \text{ if }\ n\in F ,\\\frac{1}{4}&\ \text{ if } \ n\notin F .\end{cases}
$$
\end{lemma}
\begin{proof}
We are going to combine a construction of  Berend (Ex $7.1$ in \cite{Bere85}) with  a construction of Furstenberg (page $40$ in \cite{Fu81}).

Suppose first that the range of  both  sequences  misses infinitely many integers.
Let $X=\{0,1\}^\Z$,  $\mu$ be the $(1/2,1/2)$ Bernoulli measure on $X$, and  $T$ be the shift transformation. Given a permutation $\pi$ of $\Z$ with $\pi(0)=0$ we define the measure preserving transformation  $\psi_{\pi}\colon X\to X$ by
$$(\psi_{\pi} x)_n=
\begin{cases}
x_{0}& \ \text{ if } \  n=0,  \\
1-x_{\pi(n)}& \ \text{ if } \   n\neq 0.
\end{cases}$$
 Let $S=\psi_{\pi}^{-1}T\psi_{\pi}$ ($S$ is  also  Bernoulli). Since $(\psi_\pi)^{-1}=\psi_{\pi^{-1}}$ and $(\psi_{\pi^{-1}}x)_0=x_0$, for $n\in \N$ we have
$$
(S^nx)_0=(\psi_{\pi^{-1}}T^n \psi_{\pi} x)_0=(T^n \psi_{\pi}x)_0=(\psi_{\pi} x)_n =
1-x_{\pi(n)}.
$$
Hence, if $A=\{x\in X \colon x(0)=1\}$ we have
$$
T^{-a(n)}A\cap S^{-b(n)}A
=
\{x\in X\colon x_{a(n)}=1, x_{\pi(b(n))}=0\}.
$$
Finally,  we make an appropriate choice for $\pi$.
Since the sequences  $(a(n))$ and $(b(n))$ are injective and miss infinitely many integers, there exists a permutation $\pi$ of the integers   that fixes $0$ and satisfies
  $\pi(b(n))=a(n)$ if $n\in F$, and $\pi(b(n))\neq a(n)$ if $n\notin F$. Then
$$
\mu(T^{-a(n)}A\cap S^{-b(n)}A)=\begin{cases} 0 &\ \text{ if }\  n\in F,\\1/4  &\ \text{ if } \ n\notin F. \end{cases}
$$

 We now consider the general case. Notice that the range of the sequences $(2a(n))$ and $(2b(n))$ misses infinitely many values.
 We consider the transformations $T^2$ and $S^2$ in place of $T$ and $S$ (again they are Bernoulli) and carry out the previous argument with a permutation $\pi$ that satisfies $\pi(2b(n))=2a(n)$ if $n\in F$ and $\pi(2b(n))\neq 2a(n)$ if $n\notin F$.
\end{proof}

\begin{corollary}
Let  $a,b\colon \N\to\Z\setminus\{0\}$ be  injective sequences,
and $c\colon \N\to [0,1/4] $ be any sequence.
Then there exist   a probability space $(X,\X,\mu)$,  measure preserving transformations $T,S\colon X\to X$, and $A\in \X$, such that    for every $n\in \N$ one has
$$
c(n)=\mu\big(T^{-a(n)}A\cap S^{-b(n)}A\big).
$$
\end{corollary}
\begin{proof}
The set $S$ consisting of all sequences that take values on a set $[0,\alpha]$, where $\alpha>0$,  is a compact (with the topology of pointwise convergence)  convex subset of the locally convex space that consists  of all bounded sequences. The extreme points of  $S$  are the  sequences that take values in the set $\{0,\alpha\}$. The set $\text{ext}(S)$, of extreme points of $S$, is closed, hence, by the theorem of Krein-Milman, every element in $S$ is the barycenter of a Borel probability measure on $\text{ext}(S)$.
As a consequence we get  that given any sequence $c\colon\N\to [0,1/4]$,
  there exists a Borel probability measure $\sigma$ on a compact metric space $(Y,d)$,
and sequences  $c_y\colon \N\to \{0,1/4\}$, $y\in Y$,
  such that for every $n\in \N$ one has
$$
c(n)=\int c_y(n) \ d\sigma(y).
$$
 Looking at the proof of Lemma~\ref{L:NonRecConv} we see that there exist  measure preserving transformations $T_y$ and $S_y$,  acting on the same probability space $(X,\X,\mu)$, and $A\in \X$, such that for every $y\in Y$ and $n\in \N$ one has
$$
c_y(n)=\mu\big(T_y^{-a(n)}A\cap S_y^{-b(n)}A\big).
$$
On the space $(Y\times X, \mathcal{B}_Y\times \X, \sigma\times \mu)$ we define the measure preserving transformations $T,S\colon Y\times X\to Y\times X$ by the formula $T(y,x)=(y,T_y(x))$ and $S(y,x)=(y, S_y(x))$. Then for every $n\in \N$ one has
$$
\mu\big(T^{-a(n)}A\cap S^{-b(n)}A\big)=\int \mu\big(T_y^{-a(n)}A\cap S_y^{-b(n)}A\big) \ d\sigma(y)=
\int c_y(n) \ d\sigma(y)=c(n).
$$
\end{proof}

\begin{proof}[Proof of Theorem~\ref{T:NonRec-NonConv}]
For non-convergence take  $F=\bigcup_{n\in \N}[2^{2n},2^{2n+1}]$ in Lemma~\ref{L:NonRecConv} and define $f=g={\bf 1}_A$.
For non-recurrence take $F=\N$ in Lemma~\ref{L:NonRecConv}.
\end{proof}

\section{Appendix}
We prove some results that were used in the main part of the article.
\subsection{Lacunary subsequence trick.} We are going to give
a variation of a trick that is often used to prove convergence results for averages (see \cite{RW95} for several such instances).
\begin{lemma}\label{L:Farao}
Let $(a_n)_{n\in\N}$ be a sequence of non-negative real numbers  and $(W_n)_{n\in\N}$ be an increasing sequence of positive real numbers  that satisfies
$$
\lim_{\gamma\to 1^+}\limsup_{n\to\infty} \frac{W_{[\gamma n]}}{W_n}=1.
$$
 For
  $N\in \N$ let
$$
A_N\mathrel{\mathop:}=\frac{1}{W_N}\sum_{n=1}^N a_n.
$$
Suppose that there exists $L\in [0,+\infty]$, and a sequence
of real numbers $\gamma_k\in (1,+\infty)$, with $\gamma_k\to 1$, and
such that for every $k\in \N$ we have
$$
\lim_{N\to\infty} A_{[\gamma_k^N]}=L.
$$
Then
$$
\lim_{N\to\infty} A_{N}=L.
$$
\end{lemma}
\begin{proof}
Fix $k\in \N$ and for  $N\in \N$ let $M=M(k,N)$ be a non-negative integer  such that
$$
[\gamma_k^M]\leq N< [\gamma_k^{M+1}].
$$
 Since $a_n\geq 0$ for every $n\in \N$ and $W_n$ is increasing, we have
$$
A_N=\frac{1}{W_N}\sum_{n=1}^N a_n\leq \frac{1}{W_{[\gamma_k^{M}]}}\sum_{n=1}^{[\gamma_k^{M+1}]}a_n\leq c_{k,M} A_{[\gamma_k^{M+1}]} \quad \text{ where } \quad c_{k,M}\mathrel{\mathop:}= W_{[\gamma_k^{M+1}]}/W_{[\gamma_k^M]}.
$$
Similarly we have
$$
A_N\geq  c_{k,M}^{-1} A_{[\gamma_k^{M}]}.
$$
Putting the previous estimates together we get
\begin{equation} \label{E:980}
c_{k,M}^{-1}  A_{[\gamma_k^{M}]} \leq  A_N\leq c_{k,M}  A_{[\gamma_k^{M+1}]}.
\end{equation}
Notice that our assumptions give that
\begin{equation}\label{E:981}
\lim_{k\to\infty }\limsup_{M\to\infty} \ c_{k,M}=1.
\end{equation}
Since $M=M(k,N)\to \infty$ as $N\to \infty$ and $k$ is fixed, letting $N\to \infty$ and then $k\to\infty$ in \eqref{E:980}, and combining equation \eqref{E:981}
  with our assumption
$\lim_{N\to\infty} A_{[\gamma_k^N]}=L$,
we deduce that
$$
\liminf_{N\to\infty} A_N=\limsup_{N\to \infty} A_N=L.
$$
This completes the proof.
\end{proof}

\begin{corollary}\label{C:lacunary}
Let $(X,\X,\mu)$ be a probability space,
 $f_n\colon X\to \R$, $n\in\N$, be non-negative  measurable functions, $(W_n)_{n\in\N}$ be as in the previous lemma,
 and for $N\in \N$ let
 $$
 A_N(x)\mathrel{\mathop:}=\frac{1}{W_N}\sum_{n=1}^N f_n(x).
 $$
 Suppose that there exists a  function $f\colon X\to \R$ and a sequence of real numbers $\gamma_k\in (1,\infty)$, with $\gamma_k\to 1$, and such that for every $k\in \N$ we have
 for almost every $x\in X$ that
\begin{equation}\label{E:412}
\lim_{N\to\infty}A_{[\gamma_k^N]}(x)=f(x).
\end{equation}
Then
$$
\lim_{N\to\infty} A_N(x)=f(x) \quad \text{ for almost every } x\in X.
$$
\end{corollary}
\begin{proof}
It suffices to notice that for  almost every $x\in X$
equation \eqref{E:412} is satisfied for  every $k\in\N$, and then apply Lemma~\ref{L:Farao}.
\end{proof}

\subsection{Weighted averages}
The following lemma is classical and can be  proved using summation by parts (also the assumptions on the weights $w_n$ can be weakened).
 \begin{lemma}\label{L:ChangeVar}
   Let $(v_n)_{n\in\N}$ be a  sequence of vectors in a normed space,  $(w_n)_{n\in \N}$ be a
  decreasing  sequence of positive real numbers that satisfies
 $w_n\sim n^{-a}$  for some $a\in (0,1)$, and
for $N\in \N$ let $W_N\mathrel{\mathop:}=w_1+\cdots+w_N$.
   Then the averages $\frac{1}{N}\sum_{n=1}^N v_n$ and
   the averages $ \frac{1}{W_N}\sum_{n=1}^N w_n\ \! \!v_n$ are asymptotically equal.
   \end{lemma}
\subsection{Van der Corput's lemma} We state a variation of
a classical elementary estimate of van der Corput.
\begin{lemma} \label{L:VDC2} Let  $V$ be an inner product space, $N\in \N$,
and  $v_1,\ldots, v_N\in V$.
Then for every integer  $M$ between $1$ and $N$ we have
$$
\norm{\sum_{n=1}^N v_n}^2\leq 2M^{-1}N\cdot \sum_{n=1}^N\norm{v_n}^2+
 4M^{-1} N\sum_{m=1}^M \Big| \sum_{n=1}^{N-m}<v_{n+m}, v_n>\Big|.
$$
\end{lemma}
In the case where $V=\R$ and $\norm{\cdot}=|\cdot|$, the proof can be found, for example,  in \cite{KN74}.
The proof in the general case is essentially identical.

\subsection{Borel-Cantelli in density}
We are going to use the following Borel-Cantelli type lemma:
\begin{lemma}\label{L:BorCan}
  Let  $E_n$, $n\in \N$, be events on a probability space $(\Omega,\mathcal{F}, \P)$ that
  satisfy $\P(E_n)\ll (\log n)^{-1-\varepsilon}$ for some $\varepsilon>0$.
  Then almost surely
the set $\{n\in \N \colon \omega \in E_n\}$ has zero density.\footnote{On the other hand, it is not hard to construct a probability space $(\Omega,\mathcal{F}, \P)$  and  events $E_n$, $n\in \N$, such that
$\P(E_n)\leq (\log n)^{-1}$, and almost surely the set $\{n\in \N \colon \omega \in E_n\}$
has positive upper density.}
\end{lemma}
\begin{proof}
 Let
$$
A_N(\omega)\mathrel{\mathop:}=\frac{1}{N}\sum_{n=1}^N{\bf 1}_{E_n}(\omega).
$$
Our assumption gives
$$
\E_\omega(A_N(\omega))\ll (\log N)^{-1-\varepsilon}.
$$
Therefore, for every $\gamma>1$
$$
\sum_{N=1}^\infty A_{[\gamma^N]}(\omega)<+\infty
$$
almost surely.
This implies that for every $\gamma>1$
$$
\lim_{N\to\infty} A_{[\gamma^N]}(\omega)= 0 \quad \text{ almost surely}.
$$
Since $\gamma>1$ is arbitrary we conclude by Corollary~\ref{C:lacunary}   that
$$
\lim_{N\to\infty} A_N(\omega) =0 \quad \text{ almost surely}.
$$
This proves the advertised claim.
\end{proof}

\subsection{Estimates for sums of random variables.}
We use some straightforward moment estimates to get two bounds for sums of independent  random variables that were used in the proofs.
\begin{lemma}\label{L:estimate1} Let $X_n$ be  non-negative, uniformly bounded
 random variables,
 with $\E_\omega(X_n)\sim n^{-a}$ for some $a\in (0,1)$. Suppose that the random
 variables $X_n-\E_\omega(X_n)$ are orthogonal.
 Then almost surely we have
$$
\lim_{N\to\infty} \frac{1}{W_N}\sum_{n=1}^N X_n = 1,
$$
where, as usual, $W_N\mathrel{\mathop:}=\sum_{n=1}^N\E_\omega(X_n)$.
\end{lemma}
\begin{remark}
Assuming independence,  one can use Kolmogorov's three series theorem to show that the
stated result holds under the relaxed assumption $W_N\to \infty$.
\end{remark}
\begin{proof}
We can assume that $X_n(\omega)\leq 1$ for every $\omega\in \Omega$ and $n\in \N$.
We let
$$
A_N\mathrel{\mathop:}=\frac{1}{W_N}\sum_{n=1}^N Y_n
$$
where
$$
Y_n\mathrel{\mathop:}=X_n-\E_\omega(X_n).
$$
Since $Y_n$ are zero mean orthogonal random variables and $\E_\omega(Y_n^2)\leq \E_\omega(X_n)$,  we have
$$
\E_\omega(A_N^2)=\frac{1}{W_N^2}\sum_{n=1}^N\E_\omega(Y_n^2)\leq \frac{1}{W_N}.
$$
Combining this estimate with the fact  $W_N\sim N^{1-a}$, we conclude that for every $\gamma>1$ we have
$$
\sum_{N=1}^\infty \E_\omega(A_{[\gamma^N]}^2)<+\infty.
$$
Therefore, for every $\gamma>1$ we have
$$
\lim_{N\to\infty} A_{[\gamma^N]}= 0 \quad \text{almost surely},
$$
or equivalently, that
$$
\lim_{N\to\infty} \frac{1}{W_{[\gamma^N]}}\sum_{n=1}^{[\gamma^N]} X_n = 1  \quad \text{almost surely}.
$$
Since the sequence $(W_n)_{n\in \N}$ satisfies the assumptions of Corollary~\ref{C:lacunary}, and $X_n$ is non-negative,  we conclude that
$$
\lim_{N\to\infty} \frac{1}{W_N}\sum_{n=1}^N X_n = 1  \quad \text{almost surely}.
$$
This completes  the proof.
\end{proof}

\begin{lemma}\label{L:estimate2}
 Let $X_n$ be  independent,  uniformly bounded
 random variables,
 with $\E_\omega(X_n)\sim n^{-a}$  for some $a\in (0,1/6)$, and let $b$ be any positive real number.
 Then almost surely we have
$$
\Big|\sum_{m=1}^{N^b}\sum_{n=1}^N X_{n+m} X_n\Big|\ll_\omega N^{b+1-2a}.
$$
\end{lemma}
Using a lacunary subsequence trick, similar to the one used in the proof of Lemma~\ref{L:estimate1},
one can show that the conclusion actually holds for every $a\in (0,1/4)$.
\begin{proof}
Let
$$
S_N\mathrel{\mathop:}=\sum_{m=1}^{N^b}\sum_{n=1}^N (X_{n+m} X_n-\E_\omega(X_{n+m})\cdot \E_\omega(X_n))
 $$
 and
 $$
  \quad A_N\mathrel{\mathop:}=N^{-c}S_N \quad \text{ where } \quad c\mathrel{\mathop:}=b+1-2a,
$$
Since
$$
N^{-c} \sum_{m=1}^{N^b}\sum_{n=1}^N \E_\omega(X_{n+m})\cdot \E_\omega(X_n)
\ll 1,
$$
it suffices to show that almost surely we have $\lim_{N\to\infty}A_N= 0$.

Expanding $S_N^2$   and using the independence of the random variables $X_n$, we see that
$$
\E_\omega(S_N^2)\ll |\{(m,m',n,n')\in [1,N^b]^2\times [1,N]^2\colon n,n',n+m,n'+m' \text{ are not distinct }\}|\ll N^{1+2b}.
$$
Therefore,
$$
\E_\omega(A_N^2)\ll N^{-(1-4a)}.
$$
It follows that if
$k\in \N $ satisfies $k(1-4a)>1$, then
 $$
 \sum_{N=1}^\infty \E_\omega(A_{N^k}^2)<+\infty.
 $$
As a consequence,
$$
\lim_{N\to\infty} A_{N^k}= 0 \quad \text{ for every } k\in \N \  \text{ satisfying } \ k(1-4a)>1.
$$
 For any such $k\in \N$, and for a given $N\in \N$, let  $M\in \N$ be an integer such that $M^k\leq N \leq (M+1)^k$.
Then
\begin{align*}
|A_{N}-A_{M^k}| & \leq |(N^{-c}M^{kc}-1)A_{M^k}|+N^{-c}\sum_{M^k<n\leq (M+1)^k}|Y_n|\\ &\ll
|(N^{-c}M^{kc}-1)A_{M^k}|+N^{-c} M^{k-1}.
\end{align*}
The first term converges almost surely to zero as $N\to \infty$, since  this is the case for $A_{M^k}$ and
$N^{-1}M^k\leq 1$. The second term converges  to zero if
$kc>k-1$, or equivalently, if $k(2a-b)<1$.

Combining the above estimates, we get almost surely that $\lim_{N\to\infty}A_N= 0$, provided that    there exists $k\in \N$ such that
 $k(2a-b)<1<k(1-4a)$. If $a<1/6$, then  $k=3$ is such a value. This  completes the proof.
\end{proof}

\end{document}